\documentclass{scrartcl}

\usepackage{amsmath}
\usepackage{amsfonts}
\usepackage{amssymb}

\usepackage{amsthm}
\newtheorem{definition}{Definition}
\newtheorem{proposition}{Proposition}

\usepackage{tikz}

\usepackage{subdepth}
\usepackage{microtype}

\usepackage{hyperref}

\newcommand\id{\mathrm{id}} 
\newcommand\ps{\mathcal P} 
\newcommand\diset[2]{\binom{#1}{#2}} 
\renewcommand\P{\mathbf P} 
\newcommand\C{\mathbf C} 
\newcommand\B{\mathbf B} 
\newcommand\R{\mathbb R} 

\newcommand\Learn{\mathbf{Learn}} 
\newcommand\OG{\mathbf{Game}} 

\title{From open learners to open games}
\subtitle{Technical report}
\author{Jules Hedges}
\date{}

\begin{document}

\maketitle

\begin{abstract}
	The categories of open learners (due to Fong, Spivak and Tuy\'eras) and open games (due to the present author, Ghani, Winschel and Zahn) bear a very striking and unexpected similarity. The purpose of this short note is to prove that there is a faithful symmetric monoidal functor from the former to the latter, which means that any supervised neural network (without feedback or other complicating features) can be seen as an open game in a canonical way.  Roughly, each parameter is controlled by a different player, and the game's best response relation encodes the dynamics of gradient descent. We suggest paths for further work exploiting the link.
\end{abstract}

\section{Introduction}

We give an overview of the motivation for open learners and open games, but for the formal definitions and more context we refer the reader to \cite{fong-spivak-tuyeras-backprop-as-functor} for open learners and \cite{hedges_etal_compositional_game_theory} for open games.
We will use the notation of both of those papers.
More detail on open games can be found in \cite{hedges_towards_compositional_game_theory,hedges_morphisms_open_games}.

An open learner $X \to Y$ is a supervised learning system that learns a function $X \to Y$ by being presented with a sequence of pairs $(x, y)$.
This function is determined by a set of parameters, which are updated each time a new $(x, y)$ pair is presented.
Open learners are very general, but contain neural networks (at least, those consisting simply of a finite sequence of layers) as a special case by explicitly encoding backpropagation and gradient descent.
Open learners form the morphisms of a symmetric monoidal category $\Learn$ whose objects are sets, where categorical composition and monoidal product correspond to sequential and parallel composition of learning systems, which includes end-to-end and side-by-side composition of neural networks as a special case.
The resulting compositionality of the backpropagation and gradient descent semantics is the main motivation for studying open learners.

An open game $\diset X S \to \diset Y R$ is a fragment of a non-cooperative game that starts in an initial state $X$ that players can (potentially) observe, and with their choices determining a final state in $Y$.
Players act rationally (in the sense of classical game theory) in order to optimise a value in $R$ (which is typically $\mathbb R^n$ with each player attempting to maximise one coordinate), and a `co-value' in $S$ is passed back to act as the value for players in the past.
Instead of parameters we have strategies for players, and instead of updating parameters we have a `best response' relation on strategy profiles.
In general this cannot be viewed as a dynamics (in the sense that iterating it is not an interesting thing to do), but we are interested in fixpoints of it, which are Nash equilibria, or mutually non-self-defeating choices of strategies.
Open games form the morphisms of a symmetric monoidal category $\OG$ whose objects are pairs of sets, where categorical composition and monoidal product correspond to sequential and parallel play.

Whereas parameter updating of an open learner takes place in the `context' of a pair $(x, y)$, best response of an open game takes place in the context of a pair $(x, k)$ where $x \in X$ is the initial state and $k : Y \to R$ is the `continuation'.
This $k$ contains strictly more information than a single point $y$; specifically it encodes \emph{counterfactuals}, saying what the payoffs would be if some other choices had been made.
From a categorical point of view, this prevents our functor from being full.
We suggest that this is the sole significant distinction between (elementary) machine learning and (elementary) game theory.

\paragraph{Acknowledgements.} The proof of functorality was sketched by the author and David Spivak around a year before this note was written.
The author would also like to thank too many people to name for discussions about this topic, after spending a year telling anybody who would listen about it, with special mention for Neil Ghani, Mike Johnson and Eliana Lorch.

\section{From open learners to open games}

Recall the definition of an open learner $X \to Y$ \cite[definition 2.1]{fong-spivak-tuyeras-backprop-as-functor} and the definition of an open game $\diset X S \to \diset Y R$  \cite[definition 3.1]{hedges_etal_compositional_game_theory}.

\begin{definition}
Given an open learner $A = (P_A, I_A, U_A, r_A) : X \to Y$, we define an open game $F (A) : \diset X X \to \diset Y Y$ as follows:
\begin{itemize}
	\item The set of strategy profiles is $\Sigma_{F (A)} = P_A$
	\item The play function is $\P_{F (A)} = I_A$
	\item The coplay function is $\C_{F (A)} = r_A$
	\item The best response relation $\B_{F (A)} (h, k) = \{ (p, U_A (p, h, k (I_A (p, h)))) \mid p \in P_A \}$
\end{itemize}
\end{definition}

Notice that there is an exact correspondence between parameters/strategies, implementation/play and request/coplay.
The only nontrivial part of this definition is converting update to best response, which also uses implementation.
This means that throughout this note there is nothing to prove except for the final case, which nevertheless takes some work.

Notice in particular that the relation $B_{F (A)} (h, k)$ is always a functional relation, i.e. every $p \in P_A$ is related to exactly one thing, namely $U_A (p, h, k (I_A (p, h)))$.

In both $\Learn$ and $\OG$, morphisms are formally defined as equivalence classes.\footnote{Both open learners and open games naturally form monoidal bicategories with the monoidal categories being the result of quotienting by invertible 2-cells. Both \cite{fong-spivak-tuyeras-backprop-as-functor} and \cite{hedges_etal_compositional_game_theory} mention this explicitly, and for open games the details are worked out in \cite{hedges_morphisms_open_games}.}
Recall the definitions of equivalence of open learners \cite[section 2]{fong-spivak-tuyeras-backprop-as-functor} and equivalence of open games \cite[definition 5.2]{hedges_etal_compositional_game_theory}.
There is a small error in the definition of equivalence of open learners as currently written in \cite{fong-spivak-tuyeras-backprop-as-functor}, which does not type-check.
The correct equation for update should be
\[ U' (f (p), a, b) = f (U (p, a, b)) \]

\begin{proposition}
	$F$ is well-defined on equivalence classes.
\end{proposition}

\begin{proof}
	Let $A \sim B : X \to Y$ be equivalent games, so there is a bijection $i : P_A \to P_B$ that respects implementation, update and request.
	Immediately $i : \Sigma_{F (P_A)} \to \Sigma_{F (P_B)}$ respects play and coplay, so we need only check that it respects best response.
	For $h : X$ and $k : Y \to Y$, we have that $(p, p') \in \B_{F (A)} (h, k)$ iff
	\[ p' = U_A (p, h, k (I_A (p, h))) \]
	iff
	\[ i (p') = i (U_A (p, h, k (I_A (p, h)))) = U_B (i (p), h, k (I_B (i (p), h))) \]
	(since $i$ is a bijection), iff $(i (p), i (p')) \in \B_{F (B)} (h, k)$.
	Hence $F (A) \sim F (B)$.
\end{proof}

\section{Categorical structure}

Recall the definition of identity morphisms in $\Learn$ \cite[section 2]{fong-spivak-tuyeras-backprop-as-functor} and $\OG$ \cite[definition 5.5]{hedges_etal_compositional_game_theory}.

\begin{proposition}
	$F$ takes identities in $\Learn$ to identities in $\OG$.
\end{proposition}

\begin{proof}
	Let $\id_X : X \to X$ be an identity in $\Learn$.
	Then the best response relation $\B_{F (\id_X)} : X \times (X \to X) \to \ps (1 \times 1)$ is
	\[ \B_{F (\id_X)} (h, k) = \{ (*, U_{\id_X} (*, h, k (h))) \} = \{ (*, *) \} = \B_{\id_{\diset X X}} (h, k) \]
	as required.
\end{proof}

Recall the definition of categorical composition of open learners \cite[section 2]{fong-spivak-tuyeras-backprop-as-functor} and open games \cite[definition 5.1]{hedges_etal_compositional_game_theory}.

\begin{proposition}
	$F$ defines a functor $F : \Learn \to \OG$.
\end{proposition}

\begin{proof}
	Take a composable pair of open learners, $A = (P_A, I_A, U_A, r_A) : X \to Y$ and $B = (P_B, I_B, U_B, r_B) : Y \to Z$.
	The open learner $B \circ A : X \to Z$ has update function
	\[ U_{B \circ A} : (P_A \times P_B) \times X \times Z \to P_A \times P_B \]
	\[ U_{B \circ A} ((p, q), x, z) = (U_A (p, x, r_B (q, I_A (p, x), z)), U_B (q, I_A (p, x), z)) \]
	The open game $F (B \circ A) : \diset X X \to \diset Z Z$ has best response relation
	\[ \B_{F (B \circ A)} : X \times (Z \to Z) \to \ps ((P_A \times P_B) \times (P_A \times P_B)) \]
	\begin{align*}
		&\B_{F (B \circ A)} (h, k) \\
		=\ &\{ ((p, q), U_{B \circ A} ((p, q), h, k (I_{B \circ A} ((p, q), h)))) \mid p \in P_A, q \in P_B \} \\
		=\ &\{ ((p, q), U_{B \circ A} ((p, q), h, k (I_B (q, I_A (p, h))))) \mid p \in P_A, q \in P_B \} \\
		=\ &\{ ((p, q), (U_A (p, h, r_B (q, I_A (p, h), k (I_B (q, I_A (p, h))))), \\
		&\hspace{43pt} U_B (q, I_A (p, h), k (I_B (q, I_A (p, h)))))) \mid p \in P_A, q \in P_B \}
	\end{align*}
	
	On the other hand, the open game $F (A) : \diset X X \to \diset Y Y$ has best response relation
	\[ \B_{F (A)} : X \times (Y \to Y) \to \ps (P_A \times P_A) \]
	\[ \B_{F (A)} (h, k) = \{ (p, U_A (p, h, k (I_A (p, h)))) \mid p \in P_A \} \]
	and the open game $F (B) : \diset Y Y \to \diset Z Z$ has best response relation
	\[ \B_{F (B)} : Y \times (Z \to Z) \to \ps (P_B \times P_B) \]
	\[ \B_{F (B)} (h, k) = \{ (q, U_B (q, h, k (I_B (q, h)))) \mid q \in P_B \} \]
	Putting these together, the open game $F (B) \circ F (A) : \diset X X \to \diset Z Z$ has best response relation
	\[ \B_{F (B) \circ F (A)} : X \times (Z \to Z) \to \ps ((P_A \times P_B) \times (P_A \times P_B)) \]
	\begin{align*}
		&\B_{F (B) \circ F(A)} (h, k) \\
		=\ &\{ ((p, q), (p', q')) \mid (p, p') \in \B_{F (A)} (h, k') \text{ and } (q, q') \in \B_{F (B)} (I_A (p, h), k) \}
	\end{align*}
	where $k' : Y \to Y$ is given by $k' (y) = r_B (q, y, k (I_B (q, y)))$.
	This expands to the set of $((p, q), (p', q'))$ where
	\[ p' = U_A (p, h, r_B (q, I_A (p, h), k (I_B (q, I_A (p, h))))) \]
	and
	\[ q' = U_B (q, I_A (p, h), k (I_B (q, I_A (p, h))))) \]
	Comparing to the above, we see that $\B_{F (B) \circ F (A)} (h, k) = \B_{F (B \circ A)} (h, k)$.
\end{proof}

\begin{proposition}
	$F$ is faithful.
\end{proposition}

\begin{proof}
	Let $A, B : X \to Y$ be open learners.
	We show that if there is an equivalence of open games $F (A) \sim F (B)$ then there is an equivalence of open learners $A \sim B$.
	Suppose we have such a bijection $i : P_A \to P_B$.
	It immediately respects implementation and request.
	For any $p \in P_A$, $x \in X$ and $y \in Y$, define $k : Y \to Y$ by $k (y') = y$. Then
	\[ (p, U_A (p, x, y)) = (p, U_A (p, x, k (I_A (p, x)))) \in \B_{F (A)} (x, k) \]
	Since $F (A) \sim F (B)$ it follows that $(i (p), i (U_A (p, x, y))) \in \B_{F (B)} (x, k)$.
	We also have
	\[ (i (p), U_B (i (p), x, y)) = (i (p), U_B (i (p), x, k (I_B (i (p), x)))) \in \B_{F (B)} (x, k) \]
	Since $\B_{F (B)} (x, k)$ is a functional relation, we deduce that $i (U_A (p, x,  y)) = U_B (i (p), x, y)$.
	This proves that $A \sim B$.
\end{proof}

\section{Monoidal structure}

\begin{proposition}
	For any open learners $A, B$, $F (A \otimes B) = F (A) \otimes F (B)$.
\end{proposition}

\begin{proof}
	Suppose that $A = X \to Y$ and $B : W \to Z$.
	Again, it is only necessary to check update/best response.
	
	The update function of $A \otimes B$ is
	\[ U_{A \otimes B} : (P_A \times P_B) \times (X \times W) \times (Y \times Z) \to P_A \times P_B \]
	\[ U_{A \otimes B} ((p, q), (x, w), (y, z)) = (U_A (p, x, y), U_B (q, w, z)) \]
	Then the best response relation of $F (A \otimes B)$ is
	\[ \B_{F (A \otimes B)} : (X \times W) \times (Y \times Z \to Y \times Z) \to \ps ((P_A \times P_B) \times (P_A \times P_B)) \]
	\begin{align*}
		\B_{F (A \otimes B)} ((x, w), k) =\ &\{ ((p, q), (U_A (p, x, y), U_B (q, w, z))) \\
		&\mid p \in P_A, q \in P_B, (y, z) = k (I_A (p, x), I_B (q, w)) \}
	\end{align*}
	
	On the other hand, $F (A) : \diset X X \to \diset Y Y$ has best response relation
	\[ \B_{F (A)} : X \times (Y \to Y) \to \ps (P_A \times P_A) \]
	\[ \B_{F (A)} (x, k) = \{ (p, U_A (p, x, k (I_A (p, x)))) \mid p \in P_A \} \]
	and $F (B) : \diset W W \to \diset Z Z$ has best response relation
	\[ \B_{F (B)} : W \times (Z \to Z) \to (P_B \to \ps (P_B)) \]
	\[ \B_{F (B)} (w, k) = \{ (q, U_B (q, w, k (I_B (q, w)))) \mid q \in P_B \} \]
	Then $F (A) \otimes F (B) : \diset{X \times W}{X \times W} \to \diset{Y \times Z}{Y \times Z}$ has best response function
	\[ \B_{F (A) \otimes F (B)} : (X \times W) \times (Y \times Z \to Y \times Z) \to (P_A \times P_B \to \ps (P_A \times P_B)) \]

	\[ \B_{F (A) \otimes F(B)} ((x, w), k) (\sigma, \tau) = \B_{F (A)} (x, k_1) (\sigma) \times \B_{F (B)} (w, k_2) (\tau) \]
	where $k_1 : Y \to Y$ is given by $k_1 (y) = \pi_1 (k (y, I_B (\tau, w)))$, and $k_2 : Z \to Z$ is given by $k_2 (z) = \pi_2 (k (I_A (\sigma, x), z))$.
	This expands to
	\begin{align*}
		&\{ U_A (\sigma, x, k_1 (I_A (\sigma, x))) \} \times \{ U_B (\tau, w, k_2 (I_B (\tau, w))) \} \\
		=\ &\{ U_A (\sigma, x, \pi_1 (k (I_A (\sigma, x), I_B (\tau, w)))) \} \times \{ U_B (\tau, w, \pi_2 (k (I_A (\sigma, x), I_B (\tau, w)))) \} \\
		=\ &\{ (U_A (\sigma, x, \pi_1 (k (I_A (\sigma, x), I_B (\tau, w)))), U_B (\tau, w, \pi_2 (k (I_A (\sigma, x), I_B (\tau, w))))) \}
	\end{align*}
	as required.
\end{proof}

\begin{proposition}
	$F$ is a strict symmetric monoidal functor.
\end{proposition}

\begin{proof}
	$F$ takes the structure morphisms of $\Learn$ to structure morphisms of $\OG$.
\end{proof}

In compositional game theory, particular emphasis is placed on a certain family of open games $\varepsilon_X : \diset X X \to I$ known as \emph{counits}.
Recall their definition \cite[definition 4.5]{hedges_etal_compositional_game_theory}.

\begin{proposition}
	Counits are in the image of $F$.
\end{proposition}

\begin{proof}
	For each set $X$ consider the open learner $!_X : X \to 1$ with set of parameters $P_{!_X} = 1$ and request function $r_{!_X} (*, x, *) = x$.
	(The implementation and update functions both have codomain $1$, and hence are trivial.)
	We see immediately that $F (!_X) = \varepsilon_X$.
\end{proof}

\section{Outlook}

We have presented a formal connection between machine learning and game theory, and as such, it suggests applications of each to the other.
Above all, the link is in need of attention from somebody knowledgeable about machine learning from an applied perspective.

After several pages of theory, we re-state in English what we have: a canonical way to view any (sufficiently simple) neural network (and more general learning algorithms) as a fragment of a game.
Each parameter to be learned acts as though it is controlled by one player, and the best response relation encodes parameter updating.
Note that for games in general the best response relation should not be thought of as a `dynamics', and iterating it will typically not converge to equilibrium, however games that result from learning algorithms do have an interesting best response dynamics.

\begin{enumerate}

\item What sort of games can arise from neural networks and other learning algorithms?
Can techniques from game theory become useful in their analysis?

\item A player observing a value $x$ and choosing a value $y$ in order to maximise a real number is represented by an open game of type $\diset X 1 \to \diset{Y}{\mathbb R}$.
These open games are `monolithic', i.e. they are taken as generators that can label boxes in string diagrams, and their definition involves the $\arg\max$ operator. (This is precisely the place that optimisation of real numbers enters compositional game theory.)
Is it possible to decompose players as a learning algorithm with $F$ applied, and then pre- and post-composed with some `plumbing' to give it the correct type?
For example, can we take a neural network with $m$ inputs and $n$ outputs, represented as an open learner $N : \R^m \to \R^n$, and then choose functions $f : X \to \R^m$, $g : \R^n \to Y$ and $h : \R^n \times \R \to \R^n$ (lifted to zero-player open games) such that
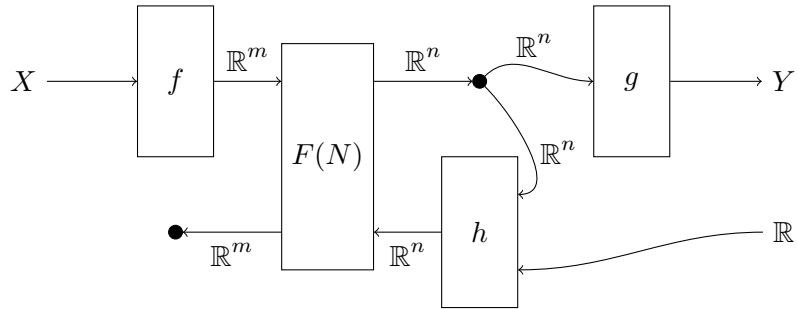
\begin{center} \begin{tikzpicture}
	\node (X) at (0, 1) {$X$};
	\node (f) [rectangle, minimum height=2cm, minimum width=1cm, draw] at (2, 1) {$f$};
	\node (N) [rectangle, minimum height=3cm, minimum width=1cm, draw] at (4, 0) {$F (N)$};
	\node (d1) [circle, scale=.5, fill=black, draw] at (2, -1) {};
	\node (d2) [circle, scale=.5, fill=black, draw] at (6, 1) {};
	\node (g) [rectangle, minimum height=2cm, minimum width=1cm, draw] at (8, 1) {$g$};
	\node (Y) at (10, 1) {$Y$};
	\node (h) [rectangle, minimum height=2cm, minimum width=1cm, draw] at (6, -1) {$h$};
	\node (R) at (10, -1) {$\R$};
	\node (dummy1) at (0, -.5) {}; \node (dummy2) at (0, -1.5) {};
	\draw [->] (X) to (f);
	\draw [->] (f) to node [above] {$\R^m$} (N.west |- f);
	\draw [->] (N.west |- d1) to node [below] {$\R^m$} (d1);
	\draw [->] (N.east |- d2) to node [above] {$\R^n$} (d2);
	\draw [->] (d2) to [out=45, in=180] node [above] {$\R^n$} (g);
	\draw [->] (g) to (Y);
	\draw [->] (d2) to [out=-45, in=0] node [right] {$\R^n$} (h.east |- dummy1);
	\draw [->] (R) to [out=180, in=0] (h.east |- dummy2);
	\draw [->] (h) to node [below] {$\R^n$} (N.east |- h);
\end{tikzpicture} \end{center}
is a useful representation of a player who can learn?

\item The author implemented in Haskell the Cournot duopoly open game from \cite[section 4.4]{hedges_etal_compositionality_string_diagrams_game_theory} but with the players' ``best response'' functions were co-opted to perform a single step of gradient descent (with a very naive numerical implementation) rather than $\arg\max$.
Iterating the resulting best response of the composite was found to converge rapidly to the market equilibrium.
This was inspired by the link with open learners, but was not formally an example of it.
Can it be made into a formal example, and does this suggest good ways to structure equilibrium-approximating programs?

\item Can this idea be used to put game theory on a more realistic foundation with players who learn in a similar way to humans, without throwing away all of the benefits of having a theory at all?

\item Does this give us a way to formalise \emph{endogenous learning agents} in an economic model \cite[section 4]{blumensath13}?

\item It is well-known that Nash equilibria and many other notions of economic equilibrium can be intractable to calculate (the field of \emph{algorithmic game theory} studies this).
Can we use the `unreasonable effectiveness' of deep learning to cheat complexity theory and efficiently approximate equilibria in real examples?

\item GANs (generative adversarial networks) have achieved enormous success for machine learning tasks such as image generation.
A GAN consists of two deep neural networks, the \emph{generator} and the \emph{discriminator}, with a strong intuition that they are playing a zero-sum game against each other \cite{goodfellow_etal_generative_adversarial_nets,oliehoek_etal_gangs}.
Can we formalise this intuition by embedding the two neural networks into a game-theoretic situation that is itself not a neural network?
(Note that \emph{substituting into a game-theoretic situation} is precisely the sort of thing that can be made precise using open games.)

\item Can we systematically design more general GAN-like machine learning systems, for example having two networks playing a non-zero-sum game, or having more than two networks interacting?
Can analysis techniques from game theory help guide us to designs that have the properties we want?

\item By designing games whose players are learning agents implementing using neural networks, can we achieve a `best of all words' hybrid theory combining game theory, multiagent systems and machine learning?

\item Finally, a purely theoretical problem: to simplify and modularise the proof given here.
Open games can be factorised in terms of lenses \cite{hedges_morphisms_open_games}, and a similar factorisation is possible for open learners (work in progress by Brendan Fong and Mike Johnson).
However the functor $F$ does not respect these two factorisations.
As a result the author was unable to find a good way to modularise the proofs given in this paper, and instead presented them monolithically and ad-hoc, which is unsatisfying.

\end{enumerate}

\bibliographystyle{alpha}
\bibliography{\string~/Dropbox/Work/refs}

\end{document}